\documentclass[12pt,reqno,
a4paper]{amsart}
\usepackage{
   latexsym, amsmath,  amsfonts, amssymb,  amsthm,   amscd,
    gensymb,  graphicx, comment,  etoolbox, url,
    booktabs, stackrel, mathtools,enumitem, mathdots,  microtype, lmodern,    mathrsfs, graphicx, tikz,  longtable,tabularx, float, tikz, pst-node, tikz-cd, multirow, tabularx, amscd,  bm, array, makecell, diagbox, booktabs,ragged2e, caption, subcaption }
\usepackage{makecell,slashbox}
\usepackage{blkarray} 
\usepackage{framed}
\usepackage{xcolor}
\usepackage[utf8]{inputenc}
\usepackage{microtype, wrapfig,textcomp,mathrsfs,csquotes,fbb}
\usepackage[colorlinks=true, linkcolor=blue, citecolor=blue, urlcolor=blue, breaklinks=true]{hyperref}
\usepackage[capitalise]{cleveref}
\usepackage{todonotes}
\usetikzlibrary{positioning}
\usetikzlibrary{shapes,arrows.meta,calc}

\usetikzlibrary{arrows}

\newtheorem{theorem}{Theorem}[section]

\newtheorem{corollary}[theorem] {Corollary}
\newtheorem{definition}[theorem]{Definition}
\newtheorem{example}[theorem]{Example}
\newtheorem{lemma}[theorem]{Lemma}

\newtheorem{proposition}[theorem]{Proposition}
\newtheorem{remark}[theorem]{Remark}

\newtheorem{question}[]{Question}

\newcommand\R{\mathbb{R}}
\newcommand\Z{\mathbb{Z}}
\newcommand\C{\mathbb{C} }

\newcommand{\TC}{\mathrm{TC}}

\newcommand{\ct}{\mathrm{cat}}
\newcommand{\cl}{\mathrm{cl}}

\newcommand{\zl}{\mathrm{zcl}}

\newcolumntype{x}[1]{>{\centering\arraybackslash}p{#1}}

\begin{document}
\title[Motion planning in Dold manifolds of real torus type]{LS-category and topological complexity of real torus manifolds and Dold manifolds of real torus type}

\author[K. Brahma ]{Koushik Brahma}
\address{Department of Pure and Applied Mathematics, Faculty of Science and Engineering, Waseda University, 3-4-1 Okubo, Shinjuku-ku, Tokyo 169-8555, Japan.}
\email{w.iac24166@kurenai.waseda.jp}

\author[N. Daundkar]{Navnath Daundkar}
\address{Department of Mathematics, Indian Institute of Technology Madras, Chennai 600036, India.}
\email{navnath@iitm.ac.in}

\author[S. Sarkar]{Soumen Sarkar}
\address{Department of Mathematics, Indian Institute of Technology Madras, Chennai 600036, India.}
\email{soumen@iitm.ac.in}


\begin{abstract}
The real torus manifolds are a generalization of small covers, and the
Dold manifolds of real torus type are a class of non-trivial fibre bundles over the projective product spaces with real torus manifolds as fibres. In this paper, first,
we compute the LS-category of these two types of manifolds and obtain sharp bounds on their topological complexities. We show that under certain hypotheses, the topological complexities of real torus manifolds of dimension $n$ are either $2n$ or $2n+1$.
We figure out tight bounds for the topological complexity of generalized real Bott manifolds, and in many cases, the difference between these upper and lower bounds is less than 5. We compute the $\mathbb{Z}_2$-equivariant LS-category of small covers when the $\mathbb{Z}_2$-fixed points are path connected. In the end, we study the symmetric topological complexity of the above-mentioned manifolds and obtain exact values for infinitely many cases.
\end{abstract}

\maketitle

\section{Introduction}
Farber \cite{Far} introduced the concept of topological complexity to study the robot motion planning problem through the topological lens. 
For a topological space $X$, the topological complexity $\TC(X)$ is a numerical homotopy invariant.
Let $X$ be a path-connected space and $PX$ be the space of all paths
in $X$ equipped with the compact open topology. 
Let $\gamma \colon [0,1]\to X$ be any path in $X$. 
Then there is a fibration $\pi \colon PX \to X\times X$ defined by $\pi(\gamma)=(\gamma(0),\gamma(1))$.
The \emph{topological complexity} of $X$ is the smallest $k$ for which $X\times X$ admits an open cover $U_1, \dots, U_k$, such that there exist continuous sections of $\pi$ on $U_{i}$ for $1\leq i \leq k$. In general, determining the exact value of  $\TC(X)$ is a hard problem, even for spaces with nice CW-complex structures like Dold manifolds. 
Since Farber first introduced this invariant in $2003$, numerous methods have been created to determine bounds on $\TC(X)$, and the precise value of this invariant has been calculated for several spaces, see for example \cite{Dra2}, \cite{Dra},  \cite{Far}, \cite{FTY}, \cite{GGTX}, \cite{GLO}, \cite{GM}. 
This invariant is closely related to the well-known invariant, the \emph{Lusternik-Schnirelmann category} (in short, \emph{LS-category}) \cite{LS}. For a path connected space $X$, the LS-category of $X$ is denoted  $\mathrm{cat}(X)$. The  $\mathrm{cat}(X)$ is the smallest integer $r$ such that $X$ can be covered by $r$ open subsets 
$V_1, \dots, V_r$ for which the inclusion $V_i\hookrightarrow{} X$ is null-homotopic for $1\leq i\leq r$. 
In particular, it was shown in  \cite{CLOT} that,
\[\ct(X)\leq \TC(X)\leq 2\ct(X)-1.\]
The cohomological methods have been used to obtain some lower bounds on $\TC(X)$ and $\ct(X)$. We describe them here briefly.
Let $R$ be a commutative ring with unity and $X$ be a path-connected topological space with its cohomology ring $H^{\ast}(X; R)$. The cup-length of $X$ over $R$ is the maximal integer $\ell$ such that there exists $x_i\in {\tilde{H}}^{\ast}(X;R)$ for $i=1, \ldots, \ell$ satisfying $\prod_{i=1}^{\ell}x_i\neq 0$. 
We denote this number by $\cl_{R}(X)$. The number gives a lower bound for the $\ct(X)$, see \cite[Proposition 1.5]{CLOT}.
Let \[\cup \colon H^{\ast}(X;R)\otimes H^{\ast}(X;R) \longrightarrow H^{\ast}(X;R)\] be the map induced by the cup product. 
Then the zero-divisors-cup-length of $X$ with respect to the coefficient ring $R$ is defined as the maximal integer $k$ such that there exist cohomology classes ${u_i\in H^{\ast}(X;R)\otimes H^{\ast}(X;R)}$ satisfying $\cup(u_i)=0$ for all $1\leq i \leq k$ and $\prod_{i=1}^{k} u_i\neq 0$. 
This integer is denoted by $\zl_{R}(X)$. It was proved in \cite[Theorem 7]{Far} that $\zl_{R}(X)$ gives a lower bound for the $\TC(X)$. 
That is, we have the following;
\begin{equation}\label{eq:lscat_tc_lbd}
 \ct(X)\geq \cl_{R}(X)+1 ~~ \text{ and } ~~ \TC(X)\geq \zl_{R}(X) +1.
\end{equation}

Several properties of generalized Bott manifolds have been studied. However, little is known about their real counterpart. Briefly, A \emph{generalized real Bott tower} of height $m$ is a sequence 
\begin{equation} \label{bott_tower}
   B_m \xrightarrow{\pi_m} B_{m-1} \xrightarrow{\pi_{m-1}} \cdots \xrightarrow{\pi_2} B_1 \xrightarrow{\pi_1} B_0= \lbrace \mbox{pt} \rbrace  
\end{equation}
of manifolds $B_j= \mathbb{P}(\underline{\mathbb{R}} \oplus E_j^{(1)} \oplus \cdots \oplus E_j^{(n_j)} )$, where $\underline{\mathbb{R}}$ is the trivial line bundle over $B_{j-1}$, $E_j^{(i)}$ is a real line bundle over $B_{j-1}$ for $i=1,\dots, n_j$, and $j= 1,\dots, m$. Here $\mathbb{P}(\cdot)$ denotes the projectivization. The space $B_j$ is called a $j$-th stage generalized real Bott manifold. Note that the fiber of the map $\pi_j$ is $\R P^{n_j}$ for $j=1, \ldots, m$. In particular, $B_1 = \R P^{n_1}$. Also, when $n_j= 1$ for every $j \in \{1, \ldots, m\}$, then $B_j$ is called an $j$th-stage real Bott manifold.  Interestingly, generalized real Bott manifolds are small covers. We note that the small covers are central objects in Toric Topology introduced by Davis and Januszkiewicz \cite{DJ}. They are smooth manifolds with effective locally standard real torus actions and have rich combinatorial properties. The concept of small covers has been extended in \cite{LM2} where the author called them $2$-torus manifolds. If the real torus action on a $2$-torus manifold is locally standard and the orbit space is a nice manifold with corners, we may call the $2$-torus manifold a `real torus manifold'. Following \cite{LM2}, one can get a combinatorial pair $(P, \lambda)$ from a real torus manifold where $P$ is the orbit space which is a nice manifold with corners, and $\lambda$ is called a $\Z_2$-characteristic function on $P$. We refer to \cite{DJ, LM2, KZ16} for several topological information on small covers and real torus manifolds. We note that the article \cite{brahma2023various} calculated the LS-category of a small cover and studied the topological complexity of a specific subclass of generalized real Bott manifolds.

Farber, Tabachnikov and Yuzvinsky \cite{FTY} showed a remarkable relation between the topological complexity of a real projective space and its immersion dimension. The topological complexity of any generalized real Bott manifolds, which are smooth projective real toric manifolds, is unknown until now. We note that a real torus manifold is a finite quotient of a real moment angle manifold that has several similar topological properties with the same dimensional sphere. 
Therefore, we think that the following question is interesting and challenging. 
\begin{question}
 Is the topological complexity of a real torus manifold $M$ equal to its immersion dimension? 
\end{question} 

The other class of manifolds we are interested in are called \emph{generalized projective product spaces}. These manifolds were recently introduced by Sarkar and Zvengrowski in \cite{SZ22}, extending the concept of (generalized) Dold manifolds \cite{Dol56, NS} and projective product spaces \cite{Davis}. Let $N$ be a manifold equipped with a free $\Z_2$-action and  $M$ a $\Z_2$-manifold. Then the diagonal $\Z_2$-action on the product $M \times N$ is free. Then the orbit space $D(M, N):=(M \times N)/ \Z_2$ is a manifold. In particular, if $M$ is a real torus manifold, we call $D(M, N)$ a \emph{Dold manifold of real torus type}. Several properties of these manifolds remain to explore for many $M$ and $N$. The goal of this article is to compute the LS-category and topological complexity of real torus manifolds (which contain generalized real Bott manifolds and small covers) and Dold manifolds of real torus type. We explain the details of this in the organization.

In Section \ref{sec_lscat_tc_smcov}, we recall some basic topological properties of real torus manifolds. We prove that the LS-category of a real torus manifold $M$ is $\dim(M)+1$ if the orbit space contains a boundary of a simple polytope, see \Cref{thm: ct rTmfd}. If $M$ is an $n$-dimensional real torus manifold such that a facet of the orbit space is an $(n-1)$-simplex, then the lower bound on the topological complexity of $M$ can be given by the zero-divisors-cup-length of $\R P^{n-1}$, see \Cref{thm: tcmpchi}. In addition, if the $\Z_2$-characteristic function $\lambda$ associated with $M$ satisfies certain conditions, then the lower bound on the topological complexity of $M$ can be given by the zero-divisors-cup-length of $\R P^{n}$, see \Cref{thm: tcmpchi}. Then, we show that a certain product of a generating set of $H^{*}(M(P, \lambda); \Z_2)$ is non-zero if $P$ is a finite product of simplices, see \Cref{yj_nj_neq_0}. As a result, we use cohomological methods to obtain a sharp lower bound on the topological complexity of generalized real Bott manifolds, see \Cref{highertc_small_cover_special} and \Cref{cor: tc lb}. Then we prove that if the product of simplices contains $k$-many odd dimensional simplices, then the dimensional upper bound on the topological complexity of generalized real Bott manifolds can be improved by $k$, see \Cref{thm_some_odd_fact}. 

In \Cref{sec: cat tc drs}, we compute the $\Z_2$-equivariant LS-category of small covers when the fixed point sets are path connected, see \Cref{thm: eqctsmallcover}. For the other kinds of $\Z_2$-actions, when fixed point sets are not connected, we show that the $\Z_2$-equivariant category of small covers could be much greater than the dimension plus one, see \Cref{thm: eq cat}. In general, it remains less than the number of vertices of the corresponding polytope. \Cref{exm: eqcatM} observes that, in certain cases, it can be exactly the number of vertices of the polytope. Then we compute the LS-category of Dold manifolds of real torus type and obtain sharp bounds on their topological complexity, see \Cref{thm: ctDmp} and \Cref{thm: tcDmp}. In many specific cases, the topological complexity of these manifolds is either $a-1$, $a$, or $a+1$ for some $a \in \mathbb{N}$, see \Cref{cor: tcdm}. As an application, we obtain sharp bounds on the symmetric topological complexity of a class of real torus manifolds and Dold manifolds of real torus type (see \Cref{thm: symtcmpchi}, \Cref{thm: sytcexact}, and \Cref{thm: symtcdmp}). Moreover, in \Cref{thm: sytcexact}, we compute the exact value of the symmetric topological complexity of many real torus manifolds. 


\section{LS-category and topological complexity of real torus manifolds}\label{sec_lscat_tc_smcov}
In this section, we recall the concepts of real torus manifolds and generalized real Bott manifolds. We compute the LS-category of a class of real torus manifolds containing small covers and give some tight bounds for the topological complexity of certain real torus manifolds. We describe a nice presentation of the cohomology ring of generalized real Bott manifolds with $\Z_2$-coefficients. Then, we obtain sharp bounds for the topological complexity of these manifolds.

We recall the definition of a nice manifold with corners following \cite{Da} and \cite{BS}. The codimension function $c$ associates to a point $x=(x_1,\dots,x_n)\in (\R_{\geq 0})^n$ the number of $x_i$ which are zero. An $n$-dimensional manifold with corners is a Hausdorff, second-countable topological space together with a maximal atlas of local charts onto open subsets of $(\R_{\geq 0})^n$ such that the overlapping maps are homeomorphisms that preserve codimension function.
In an $n$-dimensional manifold with corners, a face of codimension $n$ and $1$ are called a vertex and a facet, respectively.
The vertex set and the facet set of $P$ are denoted by $V(P)$ and $\mathcal{F}(P)$, respectively.

A manifold with corners $P$ is called nice if, for every $p\in P$ with $c(p) =2$, the number of facets of $P$ that contain $p$ is also $2$. Therefore, a codimension-$k$ face of a nice manifold with corners $P$ is a connected component of the intersection of the unique collection of $k$ many facets of $P$. In particular, if a convex polytope is nice, as in the above sense, it is called a simple polytope.

Consider $\R^n$ with standard $\Z_2$-action induced by the reflections on coordinate hyperplanes. 
A smooth action of $\Z_2^n$ on an $n$-dimensional smooth manifold $M$ is called \emph{locally standard}, if for all $x\in M$ there exists a $\Z_2^n$-invariant open neighborhood $U_x$ and a diffeomorphism $\phi \colon U_x\to V$, where $V$ is $\Z_2^n$-invariant open set in $\R^n$, and an isomorphism $\delta_x: \Z_2^n\to \Z_2^n$ such that $\phi(t\cdot x)=\delta_x(t)\cdot \phi(x)$ for all $t\in \Z_2^n$ and $x\in U_x$.

\begin{definition}
A closed, connected, and smooth $n$-dimensional manifold 
$M$ with locally standard action of $\Z_2^n$ is called a real torus manifold over a nice manifold with
corners $P$ of dimension $n$, if the following conditions are satisfied:
\begin{enumerate}
\item The boundary $ \partial P$ is non-empty.
\item There is a projection map $\mathfrak{q} \colon M \to P$ constant on orbits which maps every
$\ell$-dimensional orbit to a point in the interior of an $\ell$-dimensional face of $P$.
\end{enumerate}
\end{definition}

If $P$ is a simple polytope, then the real torus manifold $M$ is called a small cover. 
Let $P$ be a nice manifold with corners, and $\mathcal{F}(P):= \{F_1, \ldots, F_r\}$. 

\begin{definition} \label{Def 2.1}
A function $\lambda \colon \mathcal{F}(P) \rightarrow \mathbb{Z}_2^n $ is called a characteristic function if the submodule of $\mathbb{Z}_2^n$ generated by $\lbrace \lambda(F_{i_1}),\ldots, \lambda(F_{i_\ell}) \rbrace$ is an $\ell$-dimensional direct summand of $\mathbb{Z}_2^n$ whenever $F_{i_1} \cap \dots \cap F_{i_\ell} \neq \emptyset$.
The vector $\lambda_i:= \lambda(F_i)$ is called the characteristic vector associated with the facet $F_i\in \mathcal{F}(P)$, and the pair $(P, \lambda)$ is called a characteristic pair.
\end{definition}

We discuss the construction of a real torus manifold from a characteristic pair $(P, \lambda)$.
For each point $p \in P$, let $F(p)$ be the unique face of $P$, which contains $p$ in its relative interior. Then $F(p)$ is a connected component of $F_{i_1} \cap \cdots \cap F_{i_k}$ for some unique facets $F_{i_1},\ldots, F_{i_k}$ of $P$. Let $G_{F(p)}$ be the subgroup of $\mathbb{Z}_2^n$ generated by $\{\lambda(F_{i_1}),\ldots,\lambda(F_{i_k})\}$. We define an equivalence relation on $P \times \mathbb{Z}_2^n$ as follows:
\begin{equation}\label{eq_sim}
(p,g) \backsim (q,h) \Leftrightarrow p=q, g^{-1}h \in G_{F(p)}.
\end{equation}
The identification space $M^n(P,\lambda): =( P\times \Z_2^n ) / \sim$ has an $n$-dimensional manifold structure with a natural locally standard $\Z_2^n$-action induced by the group operation on the second factor of $P\times \Z_2^n$. Thus $M^n(P, \lambda)$  is a real torus manifold over $P$ (see \cite{DJ} for details).
 The projection onto the first factor gives the orbit map
\begin{equation}\label{orbit map}
\mathfrak{q} \colon M^n(P,\lambda)  \to P \text{ defined by } [p,g]_{\sim} \mapsto p \nonumber,
\end{equation}
where $[p,g]_{\sim}$ is the equivalence class of $(p,g)$. 
\begin{example}\label{ex: susp pol}
Let $Q=\{(x_1, \ldots, x_n, x) \in \R^{n+1} ~|~ \sum_1^n x_i^2 +x^2 =1, x_i \geq 0 ~\mbox{for} ~ i=1, \ldots, n\}$ for $n \geq 2$. Then $Q$ is a nice manifold with corners homeomorphic to a suspension on $(n-1)$-simplex. Note that $Q$ is not a simple polytope.
The facet set of $Q$ is $\mathcal{F}(Q) =\{F_1,\dots,F_n\}$, where $$F_i := \{(x_1, \ldots, x_n, x) \in Q ~|~ x_i=0 ~\mbox{for}~ i=1, \ldots, n\}.$$ 
Define $\lambda \colon \mathcal{F}(Q) \to \Z_2^n$ by $\lambda(F_i)=(1,\dots,1,-1,1,\dots,1)$, where $-1$ is at the $i$-th component. Then $M^n(Q, \lambda)$ is a real torus manifold homeomorphic to $S^n$, which is not a small cover. 
\end{example}

Let $M$ be a real torus manifold over a nice manifold with corners $P$ of dimension $n$, and $\mathfrak{q} \colon M \to P$ be the orbit map. Let $\mathcal{F}(P) =\{F_1, \ldots, F_r\}$. Then the isotropy subgroup of $\mathfrak{q}^{-1}(F_i)$ is a subgroup of $\Z_2^n$ and isomorphic to $\Z_2$. This isotropy subgroup is uniquely determined by $\lambda_i \in \Z_2^n \setminus \{(1, \ldots, 1)\}$ for $i=1, \ldots, m$. Thus, we have an assignment $\lambda \colon \mathcal{F}(P) \to \Z_2^n$ satisfying Definition \ref{Def 2.1}. Therefore, by \cite[Lemma 3.1]{LM2}, we have the following. 

\begin{proposition}\label{prop_equivalence}
There is an equivariant homomorphism from $M^n(P, \lambda)$ onto $M$.
\end{proposition}

Now, we will compute the LS-category of real torus manifolds over the manifolds with corners and sharp bounds on their topological complexity.

\begin{theorem}\label{thm: ct rTmfd}
Let $M$ be an $n$-dimensional real torus manifold over $P$. If a connected component of $\partial P$ is a boundary of an $n$-dimensional simple polytope $Q$, then $$\ct(M) =n+1.$$
\end{theorem}

\begin{proof}
Let $v$ be a vertex in $\partial Q \subseteq \partial P$. So there exist unique facets $F_1, \ldots, F_n$ of $P$ such that $v = F_1 \cap \cdots \cap F_n$ and $F_{i_1} \cap \cdots \cap F_{i_k}$ is an $(n-k)$-dimensional simple polytope for any $\{i_1, \ldots, i_k\} \subseteq \{1, \ldots, n\}$. Let $\mathfrak{q} \colon M \to P$ be the orbit map and $x_i$ the Poincare dual of $\mathfrak{q}^{-1}(F_i)$. It follows from the given hypotheses on $P$, that the intersection $F_1 \cap \cdots \cap F_n$ is transversal. Thus, the product $x_1 \cdots x_n$ of cohomology classes is nonzero in $H^*(M; \Z_2)$. Therefore, the result follows from the cup-length lower bound  \eqref{eq:lscat_tc_lbd} and the dimensional upper bound as given in \cite[Theorem 1.7]{CLOT}. 
\end{proof}

\begin{theorem}\label{prop: lb tcmpchi}
Let $M$ be an $n$-dimensional real torus manifold over $P$ with $n >1$. Suppose $\Delta^{n-1}$ is a facet of $P$. Then, 
\begin{equation}\label{eq: zclmp}
 n \leq \ct(M) \leq n+1, ~ \mbox{and}~ n \leq \zl_{\Z_2}(\R P^{n-1}) +1\leq \TC(M)\leq 2n+1.
\end{equation}
In particular, if $n=2^s+1$, then $$2^{s+1} \leq\TC(M)\leq 2^{s+1}+3.$$
\end{theorem}

\begin{proof}
The characteristic submanifold $M_{\Delta^{n-1}}$ corresponding to the facet $\Delta^{n-1}$ is equivariantly homeomorphic to $\R P^{n-1}$. Then the natural inclusion $\iota \colon M_{\Delta^{n-1}} \hookrightarrow M(P, \lambda)$ induces a surjective homomorphism $$\iota^* \colon  H^*(M; \Z_2) \to H^*(\R P^{n-1};\Z_2).$$ Precisely, if $F_i$ is the $i$-th facet of $P$ and $x_i \in H^{1}(M;\Z_2)$ corresponds to $F_i$, then $\iota^*(x_i) \in H^1(\R P^{n-1}; \Z_2)$ corresponds to $F_i \cap \Delta^{n-1}$. Note that  $F_i \cap \Delta^{n-1}$ is empty then $\iota^*(x_i)=0$. Since $\iota^*$ is surjective, we have $$\cl_{\Z_2}(\R P^{n-1})\leq \cl_{\Z_2}(M), ~ \mbox{and} ~\zl_{\Z_2}(\R P^{n-1})\leq \zl_{\Z_2}(M).$$ Thus, the left inequality of \eqref{eq: zclmp} follows. The right inequality of \eqref{eq: zclmp} follows from \cite[Theorem 4]{Far}.

If $n=2^s +1$, then $\zl_{\Z_2}(\R P^{n-1}) \geq 2^{s+1}-1$, see \cite{FTY}. So, the particular claim follows. 
\end{proof}

\begin{theorem}\label{thm: tcmpchi}
Let $M$ be an $n$-dimensional real torus manifold over $P$ with $n >1$ and $\Delta^{n-1}$ a facet of $P$. Suppose $F_1, \ldots, F_n$ are the facets of $P$ such that $\Delta^{n-1} \cap F_i \neq \emptyset$ and $\det[\lambda(F_1) \cdots \lambda(F_n)] =1$. Then $\ct(M) = n+1$, and 
\begin{equation}\label{eq: zclmp2}
  \zl_{\Z_2}(\R P^{n}) +1\leq \TC(M)\leq 2n +1.
\end{equation}
In particular, if $n=2^s$, then $$2^{s+1}\leq\TC(M)\leq 2^{s+1}+1.$$
\end{theorem}
\begin{proof}
We can have $M = M' \# \R P^n$ following the discussion in Subsection 1.11 in \cite{DJ}. Then $H^*(\R P^n; \Z_2)$ is a subring of $H^*(M; \Z_2)$.  
Then $n+1 = \cl_{\Z_2}(\R P^n)\leq \cl_{\Z_2}(M) \leq n+1$, and the left inequality of \eqref{eq: zclmp2} follows from the observation $\zl_{\Z_2}(\R P^n)\leq \zl_{\Z_2}(M^n(P,\lambda))$.
The right inequality of \eqref{eq: zclmp2} follows from \cite[Theorem 4]{Far}.

If $n=2^s$, then $\zl_{\Z_2}(\R P^{n}) \geq 2^{s+1}-1$, see \cite{FTY}. So, the particular claim follows. 
\end{proof}

In the remaining, we focus on computing LS-category and topological complexity of generalized real Bott manifolds. We recall the results, which show that these manifolds are small covers over the product of simplices. Eminently, we obtain sharp bounds on their topological complexity.

\begin{proposition}  \cite[Corollary 4.6]{KZ16}
The $j$-th stage generalized real Bott manifold $B_j$ of the tower \eqref{bott_tower} is a small cover over $\prod_{i=1}^j \Delta^{n_i}$ where $\Delta^{n_i}$ is the $n_i$-simplex. 
\end{proposition}
The converse statement of this result also holds by the following proposition.
\begin{proposition} \cite[Proposition 2.7]{DU19} \label{small_cover_realbott} 
    Every small cover over a product of simplices is a generalized real Bott manifold.
\end{proposition}

Let $P= \prod_{j=1}^m \Delta^{n_j}$ be a product of $m$ simplices. Let $\mathcal{F}(\Delta^{n_j}):=\{F_1^{\Delta_j}, \dots, F_{n_j+1}^{\Delta_j}\}$ be the facets of $\Delta^{n_j}$ for $j=1, \dots, m$. Therefore, the facet set of $P$ is
\begin{equation}\label{Eq_facet set of P}
\mathcal{F}(P):=\{F_{r_j}^j ~|~1\leq r_j \leq n_j+1, j=1, \dots, m\},
\end{equation}
where $F_{r_j}^j:=\Delta^{n_1} \times \dots \times \Delta^{n_{j-1}} \times F_{r_j}^{\Delta_j} \times \Delta^{n_{j+1}}\times \dots \times \Delta^{n_m}.$ Let
$\mathcal{N}_j:= \sum_{s=1}^j n_s$,
for $j=1, \ldots, m$. Thus $\mathcal{N}_1=n_1$ and $\mathcal{N}_m=n$. Let us define $\mathcal{N}_0:=0$. Note that $|\mathcal{F}(P)|=n_1+1+n_2+1+\dots+n_m+1=n+m$. Define a $\mathbb{Z}_2$-characteristic function $\lambda \colon \mathcal{F}(P) \to \mathbb{Z}_2^n,$ on $\mathcal{F}(P)$ by the following.
\[[\lambda(F^1_1), \dots, \lambda(F^1_{n_1+1}), \lambda(F^2_1), \dots, \lambda(F^2_{n_2+1})\dots,\lambda(F^m_1), \dots, \lambda(F^m_{n_m+1})]\] is given by 
 \[[e_1,\dots,e_{\mathcal{N}_1}, \boldsymbol{\alpha_1},e_{\mathcal{N}_1+1},\dots,e_{\mathcal{N}_2},\boldsymbol{\alpha_2}, \dots, e_{\mathcal{N}_{m-1}+1},\dots,e_{\mathcal{N}_m},\boldsymbol{\alpha_m}].\]

The function $\lambda$ determines the following $m \times m$ vector matrix $A$.

\begin{align}
  A:=&
\begin{pmatrix}
\boldsymbol{\alpha_1^1} & \boldsymbol{\alpha_2^1} & \dots & \boldsymbol{\alpha_m^1} \\
\boldsymbol{\alpha_1^2} & \boldsymbol{\alpha_2^2} & \dots & \boldsymbol{\alpha_m^2} \\
\vdots & \vdots & \dots & \vdots \\
\boldsymbol{\alpha_1^m} & \boldsymbol{\alpha_2^m} & \dots & \boldsymbol{\alpha_m^m} \\
\end{pmatrix}_{m \times m}=\begin{pmatrix} 
  \boldsymbol{\alpha_1} & \boldsymbol{\alpha_2} & \dots & \boldsymbol{\alpha_m}
\end{pmatrix}_{1 \times m}. \nonumber
\end{align}
where $\boldsymbol{\alpha_\ell^j}= ( \alpha_{\ell_ 1}^j, \alpha_{\ell_ 2}^j, \dots, \alpha_{\ell_ {n_j}}^j )^t \in \mathbb{Z}_2^{n_j}$.

Next, we recall the cohomology ring of the generalized real Bott manifolds from \cite{DJ}. We denote the facet $F_{r_j}^j$ by $F_{\mathcal{N}_{j-1}+r_j+j-1}$. Consider an indeterminate $x_i$ corresponding to the facet $F_i$ for each $i=1,2,\dots,n+m$.
The cohomology ring of the generalized real Bott manifold is given by 
\begin{equation}\label{Eq_cohom over small cover}
H^*(M^n(P,\lambda); \Z_2) \cong \frac{\Z_2[x_1, \dots, x_{n+m}]}{\Tilde{I} + \Tilde{J}},
\end{equation}  
 where the ideal $\Tilde{I}$ is given by 
\begin{equation}\label{Eq_the ideal I}
\Tilde{I}= \big< \{ x_{\mathcal{N}_{j-1}+1} x_{\mathcal{N}_{j-1}+2} \cdots x_{\mathcal{N}_j} x_{n+j} ~|~ j=1, \dots, m \}\big>,
\end{equation} 
 and the ideal $\Tilde{J}$ is generated by the coordinates of an $n$-tuple
\small{\begin{equation}\label{eq_lamda_j}
   \Lambda_{\Tilde{J}}=\begin{pmatrix}
   \lambda(F_1)^t & \lambda(F_2)^t & \dots & \lambda(F_{n+m})^t
   \end{pmatrix}_{(n \times (n+m))} \cdot \begin{pmatrix}
   x_1 & x_2 & \dots & x_{n+m}
   \end{pmatrix}^t_{(n+m) \times 1}.
\end{equation}}

 The $i$-th coordinate of $\Lambda_{\Tilde{J}}$ is given by
  $$x_i+\alpha_{1k_j}^j x_{n+1} + \alpha_{2k_j}^j x_{n+2} + \dots + \alpha^j_{mk_j} x_{n+m},$$
  where $i=\mathcal{N}_{j-1}+k_j$; $k_j=1, \dots, n_j$ and $j=1, \dots, m$.
  Thus any $x_i$ can be written as a $\Z_2$-linear combination of $x_{n+1}, \dots, x_{n+m}$ for $i=1, \dots, n$. For simplicity, we denote the indeterminate $x_{n+j}$ by $y_j$ for $j=1, \dots, m$. Thus,
\begin{equation} \label{x_i_relation_with_y_i}
  x_i= \sum_{\ell=1}^m \alpha^j_{\ell k_j} y_{\ell} ~\mbox{where}~ i=\mathcal{N}_{j-1}+k_j, k_j=1, \dots, n_j ~\mbox{and}~ j=1, \dots, m,
\end{equation} 
  in $H^*(M^n(P, \lambda); \Z_2)$. Therefore, using the generators of $\Tilde{I}$ we can express the cohomology ring of the generalized Bott tower $ H^*(M^n(P, \lambda); \Z_2)$ as a quotient of the polynomial ring $\Z[y_1,y_2,\dots,y_n]$. Moreover, the generators of the ideal $\Tilde{I}$ in \eqref{Eq_the ideal I} can be described in terms of $y_j$'s.

  On the other hand, using the arguments in \cite[Proposition $5.1$]{CMS10}, the matrix $A$ is conjugate to a unipotent lower triangular vector matrix of the following form:
\begin{align}\label{lower_triangular_matrix}
\Tilde{A}:=&
\begin{pmatrix}
\textbf{1} & \boldsymbol{0} & \dots & \boldsymbol{0} \\
\boldsymbol{\beta_1^2} & \textbf{1} & \dots & \boldsymbol{0} \\
\vdots & \vdots & \dots & \vdots \\
\boldsymbol{\beta_1^m} & \boldsymbol{\beta_2^m} & \dots & \textbf{1} \\
\end{pmatrix}_{m \times m}, 
\end{align}
where $\boldsymbol{\beta_j^k}= ( \beta_{j_1}^k, \beta_{j_2}^k, \dots, \beta_{j_{n_k}}^k )^t \in \mathbb{Z}_2^{n_k}$ and $\textbf{1}=(1, \dots, 1)^t \in \Z^{n_k}_2$ for $k=1, \dots, m$. The matrix $\Tilde{A}$ is called the Bott matrix. Thus, the ideal $\Tilde{J}$ becomes
  \[\Tilde{J}=\big<x_{\mathcal{N}_{j-1}+k_j}+ \beta_{1k_j}^jy_1+ \beta_{2k_j}^jy_2+ \cdots + \beta_{(j-1)k_j}^jy_{j-1}+y_j~|~ j=1,2,\dots,m\big>.\]
In the cohomology ring $ H^*(M^n(P, \lambda); \Z_2) $, we get $x_1= x_2= \cdots= x_{\mathcal{N}_1}= y_1$. 
For $j\geq 2$,
$$x_{\mathcal{N}_{j-1}+k_j}= \beta_{1k_j}^jy_1+ \beta_{2k_j}^jy_2+ \cdots + \beta_{(j-1)k_j}^jy_{j-1}+y_j \text{ for } 1\leq k_j\leq n_{j}.$$ 
Therefore, we have the following presentation of the cohomology ring with $\Z_2$ coefficients.

\begin{proposition}[\cite{brahma2023various}\label{compactform_cohomo_ring}]
The cohomology ring of the generalized Bott tower $M^n(P, \lambda)$ with $\Z_2$ coefficients is given by 
$$ H^*(M^n(P, \lambda); \Z_2) \cong \Z_2[y_1, y_2, \dots, y_m]/I,$$
where $$
I=\big< \Gamma_1,\Gamma_2,\dots,\Gamma_m\big>
\text{ and } {\Gamma}_j= \prod_{k_j=1}^{n_j} \big(y_j+ \sum_{\ell=1}^{j-1} \beta_{\ell _{k_j}}^j y_{\ell} \big) y_j \text{~ for } j=1,2,\dots, m.$$  
\end{proposition}

For the $2$-stage Bott tower over $\Delta^{n_1}\times \Delta^{n_2}$, in \cite[Theorem 4.6]{brahma2023various} it was shown that $y_1^{n_1}y_2^{n_2}\neq 0$. We now generalize this to any $m$-stage Bott tower for arbitrary $m$. This will be used in obtaining a sharp cohomological lower bound on the topological complexity of generalized Bott towers.

\begin{lemma} \label{yj_nj_neq_0}
   Let  $M^n(P, \lambda)$ be a small cover over the product of simplices $\prod_{j=1}^m \Delta^{n_j}$ with the characteristic function $\lambda$. Then, for $j \in \lbrace 1,\dots,m \rbrace$, $y_j^{n_j} \neq 0$ in the cohomology ring $H^*(M^n(P, \lambda); \Z_2)$. Moreover,  $y_1^{n_1}y_2^{n_2} \cdots y_m^{n_m} \neq 0$.
\end{lemma}

\begin{proof}
Let $M^n(P,\lambda)$ be the small cover over the product of simplices $P= \prod_{j=1}^m \Delta^{n_j}$ and $\Tilde{A}$ be the Bott matrix as in \eqref{lower_triangular_matrix} corresponding to the characteristic function $\lambda$ on $P$. We recall the cohomology ring of a small cover over a product of simplices from  \Cref{compactform_cohomo_ring}.

Now we look at the generators $\Gamma_j$'s of the ideal $I$ in the cohomology ring in 
\Cref{compactform_cohomo_ring}. 
Note that,
$$\Gamma_1= x_1 x_2 \cdots x_{n_1} y_1= y_1^{n_1+1}, ~\mbox{and}$$
 \begin{align*}
\Gamma_j 
&= (y_j+ \beta_{11}^jy_1+ \beta_{21}^jy_2+ \cdots + \beta_{(j-1)1}^jy_{j-1})(y_j+ \beta_{12}^jy_1+ \beta_{22}^jy_2+ \\
& \cdots + \beta_{(j-1)2}^jy_{j-1}) 
 \cdots (y_j+ \beta_{1n_j}^jy_1+ \beta_{2n_j}^jy_2+ \cdots + \beta_{(j-1)n_j}^jy_{j-1})y_j,
\end{align*}
for $j=2,\dots,m$. Now the least power of $y_j$ in $\Gamma_j$ is $n_j+1$. Our claim is that $y_j^{n_j} \neq 0$. 
If not, let $y_j^{n_j}= 0$. Then $y_j^{n_j} \in I$. But the least power of $y_j$, which appears as a term in a polynomial in the ideal $I$ is $y_j^{n_j+1}$. This is a contradiction. Hence $y_j^{n_j} \notin I$, i.e., $y_j^{n_j} \neq 0$ in $H^*(M^n(P, \lambda); \mathbb{Z}_2)$ for $j=1,2,\dots, m$.

Now, using the Poincare duality, we have ${x_1  \cdots x_{\mathcal{N}_1}x_{\mathcal{N}_1+1}  \cdots x_{\mathcal{N}_2} \neq 0}$. Now $x_{\mathcal{N}_1+1}=\ldots = x_{\mathcal{N}_2}$ and all of them are either $y_2$ or, $y_1+y_2$. If all of them are equal to $y_2$, then  $y_1^{n_1}y_2^{n_2}\neq 0$. If all of them are equal to $y_1+y_2$ then $$x_1  \cdots x_{\mathcal{N}_1}x_{\mathcal{N}_1+1}  \cdots x_{\mathcal{N}_2} =y_1^{n_1}(y_1+y_2)^{n_2}= y_1^{n_1}(y_2^{n_2}+y_1\cdot f(y_1,y_2))\neq 0.$$ We have $y_1^{n_1+1}=0$. 
 Therefore, $y_1^{n_1}y_2^{n_2}\neq 0$ holds for both the cases. 

A similar computation as above implies that $y_1^{n_1}y_2^{n_2+1}=0$, since ${x_1  \cdots x_{\mathcal{N}_1}\Gamma_2= 0}$. 
Again using the Poincare duality ${x_1  \cdots x_{\mathcal{N}_1}x_{\mathcal{N}_1+1}  \cdots x_{\mathcal{N}_2}\cdots x_{\mathcal{N}_1}x_{\mathcal{N}_2+1}  \cdots x_{\mathcal{N}_3} \neq 0}$ and we have two other conditions $y_1^{n_1+1}=0$ and $y_1^{n_1}y_2^{n_2+1}=0$. Using this $$y_1^{n_1}(y_1+y_2)^{n_2}(y_1+y_2+y_3)^{n_3}=y_1^{n_1}y_2^{n_2}y_3^{n_3}.$$ Then we have $y_1^{n_1}y_2^{n_2}y_3^{n_3}\neq 0$ and $y_1^{n_1}y_2^{n_2}y_3^{n_3+1}=0$. 

Now inductively, we have  $y_1^{n_1+1}=0$, $y_1^{n_1}y_2^{n_2+1}=0,\dots,y_1^{n_1}y_2^{n_2}\cdots y_{m-1}^{n_{m-1}+1}=0$. Moreover, using the Poincare duality we get, \[{x_1  \cdots x_{\mathcal{N}_1}x_{\mathcal{N}_1+1}  \cdots x_{\mathcal{N}_2}x_{\mathcal{N}_2+1} \cdots x_{\mathcal{N}_{m-1}+1} \cdots x_{\mathcal{N}_m} \neq 0}.\] 
This implies  $y_1^{n_1}y_2^{n_2} \cdots y_m^{n_m} \neq 0$.
\end{proof}

We now use \Cref{yj_nj_neq_0} to obtain the sharp lower bound on the topological complexity of generalized real Bott manifolds.

\begin{theorem}\label{highertc_small_cover_special}
Let $M^n(P, \lambda)$ be a small cover over $P= \prod_{j=1}^m \Delta^{n_j} $. If $n_j \leq 2^{r_j}-1 < 2n_j$, then \[\TC(M^n(P, \lambda)) \geq (2^{r_1}+ \cdots + 2^{r_m})- (m-1).\]
\end{theorem} 

\begin{proof}
Let  $d_j= 2^{r_j}-1$ such that $n_j \leq d_j < 2n_j$  for $j= 1,\dots, m$.
Let $\mathfrak{a}_j: =1 \otimes y_j- y_j \otimes 1$.
Now,
\begin{align*}
\mathfrak{a}_j^{d_j} &=  (1 \otimes y_j- y_j \otimes 1) ^{d_j}\nonumber \\
&= \sum_{s=0}^{d_j} (-1)^{d_j-s} \binom{d_j}{s} (1 \otimes y_j)^s (y_j \otimes 1)^{d_j-s} \nonumber\\
&=  \sum_{s=0}^{d_j} (-1)^{d_j-s} \binom{d_j}{s} (y_j^{d_j-s} \otimes y_j^s).\nonumber
\end{align*}

Note that the binomial coefficients $\binom {2^{r_j}-1} {i}$ are odd for all $0 \leq i \leq d_j$. The binomial expansion of  $\mathfrak{a}_j^{d_j}$ contains the term $(y_j^{d_j-n_j} \otimes y_j^{n_j})$ which is non-zero and there is no other same term in the expression of $\mathfrak{a}_j^{d_j}$. So $\mathfrak{a}_j^{d_j}$ is nonzero. Therefore $\mathfrak{a}_1^{d_1} \cdots \mathfrak{a}_m^{d_m}$ is non-zero.
Hence zero-divisors-cup-length of $H^*(M^n(P, \lambda); \Z_2)$ is greater than or equal to $d_1+ \cdots+ d_m $. Therefore, we have $\TC(M^n(P, \lambda)) \geq d_1+ \cdots+ d_m+1$, i.e., \[\TC(M^n(P, \lambda)) \geq 2^{r_1}+ \cdots + 2^{r_m}- (m-1).\]
This proves the claim.
\end{proof}
\Cref{highertc_small_cover_special} gives a better lower bound of the LS-category of a generalized real Bott manifold if each $n_j$ is a power of $2$.  

\begin{corollary}\label{cor: tc lb}
Let $M^n(P, \lambda)$ be a small cover over $P= \prod_{j=1}^m \Delta^{n_j} $. If $n_j=2^{r_j-1}$ for $j=1,2,\dots,m$ then  $$\TC(M^n(P, \lambda))\geq 2n-m+1.$$
\end{corollary}
\begin{proof}
If $n_j=2^{r_j-1}$ then $r_j$ satisfies $n_j\leq 2^{r_j}-1<2n_j$. Using  \Cref{highertc_small_cover_special}, we get  \[ \TC(M^n(P, \lambda))\geq 2^{r_1}+\dots+2^{r_m}-m+1.\] 
Since $n=n_1+\dots+n_m$, we have $$\TC(M^n(P, \lambda))\geq 2n-m+1.$$
\end{proof}
We obtain a sharp upper bound on $\TC(M^n(P, \lambda))$, which improves the dimensional upper bound.

\begin{theorem}\label{thm_some_odd_fact}
Let $M^n(P, \lambda)$ be a small cover over $P= \prod_{j=1}^m \Delta^{n_j} $. If $n_j > 1$ is odd for each $j \in \{i_1, \dots, i_k\} \subseteq \{1, \ldots, m\}$ then $$ \TC(M^n(P, \lambda))\leq 2n-k+1.$$
\end{theorem} 

\begin{proof}
The real moment angle manifold corresponding to the polytope $P=\prod_{j=1}^m\Delta^{n_j}$ is $\prod_{j=1}^mS^{n_j}$. Let $m_j = \frac{n_j+1}{2}$. Then, for each $j \in \{i_1, \dots, i_k\} \subset \{1, \ldots, m\}$, there is a free $S^1$-action on 
$$S^{n_j}=\{(z_1, \ldots, z_{m_j}) \in \C^{m_j} ~|~ \sum_{j=1}^{m_j} |z_j|^2 =1\}$$
defined by $\eta (z_1, \ldots, z_{m_j}) \mapsto (\eta z_1, \ldots, \eta z_{m_j})$. Thus, there exists a free $(S^1)^k$-action on $\prod_{j=1}^mS^{n_j}$. Now $M^n(P,\lambda)$ can be realized as the orbit space of the moment angle manifold
$\prod_{j=1}^mS^{n_j}$ by a free $(\Z_2)^m$-action, see \cite[Remark 2.3]{DU19}. Since both of these actions are a multiplication by a complex number, they commute and induce a free action of $(S^1/\left<e^{\pi \sqrt{-1}}\right>)^k$-action on $M^n(P, \lambda)$ if $n_j > 1$ is odd for each $j \in \{i_1, \dots, i_k\} \subset \{1, \ldots, m\}$.
Then, the result follows from
\cite[Corollary 5.3]{grantfibsymm}, since  $S^1/\left<e^{\pi \sqrt{-1}}\right> \cong S^1$.
\end{proof}

\begin{corollary}\label{cor: tc_ub_action}
Let $M^n(P, \lambda)$ be a small cover over $P= \prod_{j=1}^m \Delta^{n_j} $. If $1 < n_j \leq 2^{r_j} -1 < 2n_j$ and $n_j$ is odd for each $j \in \{i_1, \dots, i_k\} \subset \{1, \ldots, m\}$ then $$ 2^{r_1}+\dots+2^{r_m}-m+1 \leq \TC(M^n(P, \lambda))\leq 2n-k+1.$$
\end{corollary} 
\begin{proof}
This follows from Theorems \ref{highertc_small_cover_special} and \ref{thm_some_odd_fact}. 
\end{proof}

\begin{corollary}\label{cor_gap}
 Let $M^n(P, \lambda)$ be a small cover over $P= \prod_{j=1}^m \Delta^{n_j} $. If $n_j=2^{s_j}+1$ for all $j =1, \dots, k < m$ and $n_j=2^{s_j}$ for $j=k+1, \ldots, m$ where each $s_j\geq 1$, then $$2n-2k-m+1\leq \TC(M^n(P, \lambda))\leq 2n-k+1.$$
\end{corollary}
\begin{proof}
 Note that $1 \leq n_j < 2^{s_j+1}-1 =2n_j-3$ for $j=1, \ldots, k$ and $1 \leq n_j < 2^{s_j+1}-1 =2n_j-1$ for $j=k+1, \ldots, m$. Thus, using \Cref{highertc_small_cover_special}, we get the lower bound as $\sum_{j=1}^k(2n_j-3)+\sum_{j=k+1}^m(2n_j-1)+1=2n-2k-m+1$ . Also, \Cref{thm_some_odd_fact} gives the upper bound. 
\end{proof}
\begin{remark} 
If $m=3$, the difference between the upper and lower bounds is 3 and 4 when $k=0$ and $k=1$ in \Cref{cor_gap}, respectively.
\end{remark}


\section{LS-category and topological complexity of Dold manifolds of real torus type}\label{sec: cat tc drs}
In this section, we recall Dold manifolds of Bott type from \cite{SZ22} and study some of its basic properties. Then, we compute some lower and upper bounds for the topological complexity of these manifolds.

Recall the construction of small covers from Section \ref{sec_lscat_tc_smcov}. 
We now describe the involutions on small cover $M^n(P,\lambda)$ generated by $\Z_2$-subgroups of $\Z_2^n$.
Note that any $\Z_2$-subgroup $G$ of $\Z_2^n$ is given by $G=\left<(a_1,\dots,a_n)\right>$ where $a_i\in \{1,-1\}$ for $1\leq i\leq n$. 
Then define 
$\tau \colon M^n(P,\lambda)\to M^n(P,\lambda)$ by 
\begin{equation}\label{eq: invo}
 \tau([x,(t_1,\dots,t_n)])=[x, (a_1t_1,\dots,a_nt_n)],
\end{equation}
where $x \in P$, $(t_1, \ldots, t_n) \in \Z_2^n$, and $[x,(t_1,\dots,t_n)]$ is the equivalence class of the point $(x,(t_1,\dots,t_n))$ under $\sim$ in \eqref{eq_sim}.  
One can observe that the $i$-th coordinate subgroup $1\times \dots \times 1\times \Z_2\times 1\times \dots \times 1$ of $\Z_2^n$ determine a nontrivial involution on $M^n(P,\lambda)$, say $i$-th involution. 
We observe that, if $(p,t_1,\dots, t_n)\sim (p,s_1,\dots,s_n)$, then $(p,a_1t_1,\dots, a_nt_n)\sim (p,a_1s_1,\dots,a_ns_n)$ as $a_i^{-1}t_i^{-1}a_is_i=t_i^{-1}s_i$.
This shows that the relation $\sim$ is preserved under any involution generated by the $\Z_2$-subgroup of $\Z_2^n$. The following example gives involution on some $M(P, \lambda)$ outside of these types. 

\begin{example}
Let $M(P, \lambda)$ be an $n$-dimensional real torus manifold such that there is an involution $\zeta'$ on $P$ satisfying $\lambda(\zeta'(F))=\lambda(F)$ for any facet $F$ of $P$. 
Then, $\zeta'$ induces  an involution $\zeta$
on $M(P, \lambda)$ which commutes with the $\Z_2^n$-action on $M(P, \lambda)$. Let $\mathfrak{q} \colon M(P, \lambda) \to P$ be the orbit map. Suppose that the set $\{\lambda(F) ~|~ F \in \mathcal{F}(P) ~\mbox{and}~ F \cap P^{\left< \zeta' \right>} \neq 0\}$ generates $\Z_2^n$. Then one can show from the equivalence relation $\sim$ in \eqref{eq_sim} that $\mathfrak{q}^{-1}(P^{\left< \zeta' \right>})$ is path-connected. 
\end{example}

Recall the height function $\mathfrak{f}$ on a simple polytope $P$ that was introduced in the proof of Theorem 3.1 of \cite{DJ}. One can orient the edges of polytope so that the $\mathfrak{f}$ increases along them.
This orientation makes the $1$-skeleton of $P$ into a directed graph.

\begin{theorem}\label{thm: eqctsmallcover}
Let $M$ be an $n$-dimensional small cover such that $M^{\left<\tau \right>}$ is path-connected for some involution $\tau$ on $M$ and the $\left < \tau \right >$-action commutes with the $\Z_2^n$-action on $M$. Then, there exist $(n+1)$ $\left<\tau\right>$-invariant categorical sets of $M$ and 
\begin{equation}\label{eq: cattorusmfd}
 \ct_{\left< \tau \right>}(M)=n+1.  
 \end{equation} 
\end{theorem}
\begin{proof}
We have $M\cong M^n(P,\lambda)$ for a $\Z_2$-characteristic pair $(P, \lambda)$ by \Cref{prop_equivalence}.
 We also have $n+1=\ct(M^n(P,\lambda))$ by \Cref{thm: ct rTmfd}. Therefore, using \cite[Section 2]{angel2018equivariant}, one can get $n+1=\ct(M^n(P,\lambda))\leq \ct_{\left<\tau\right>}(M^n(P,\lambda))$, since the fixed point set $M^{\tau}$ is a path-connected space.

Now we describe the $\tau$-invariant categorical cover of $M^n(P,\lambda)$ containing $n+1$ many sets. For any vertex $v$ of $P$, let $U_v:=P\setminus \cup_{v\notin F}F$ be the open subset of $P$. Let $V(P)$ be the set of vertices of $P$ and 
$V_j=\{v\in V(P) ~: ~ \mathfrak{f}(v)=j\},$ where $\mathfrak{f}$ is the height function of $P$. The DJ-construction \cite[Section 1]{DJ} of $M$ gives us 
$\mathfrak{q}^{-1}(U_v)=\Z_2^n\times U_v/\sim.$ 
Let $Y_j=\cup_{v\in V_j} \mathfrak{q}^{-1}(U_v)$ for $j=0,\dots,n$. Since the actions of $\left <\tau \right >$ and $\Z_2^n$ commute, the space $Y_j$ is $\tau$-invariant. 
Observe that, \[\mathfrak{q}^{-1}(U_v)=\frac{\Z_2^n\times U_v}{\sim}\simeq \frac{\Z_2^n\times \{v\}}{(g,v)\sim (h,v)},\] for any $g,h\in \Z_2^n$ as $v$ is a fixed point of an action of $\Z_2^{n}$ on $M$. Observe that ${P=\cup_{v\in V(P)} U_v}$ and $U_v \cap U_w = \emptyset$ if $ \mathfrak{f}(v)= \mathfrak{f}(w)$.
Therefore, for each $0\leq j\leq n$, one can see that 
\begin{equation}\label{eq: inv cat cover}
Y_j\simeq \frac{\Z_2^n\times V_j}{\sim}\simeq \{[(1,\dots,1,v)] ~: ~ v\in V_j\}\subset M^n(P,\lambda)^{\tau}. 
\end{equation}
Since $M^n(P,\lambda)^{\tau}$ is path connected, ${Y_j}$ is 
$\tau$-invariantly contractible in $M$. 
This proves the first part of the claim.

Now, it is evident that the homotopy equivalence in \eqref{eq: inv cat cover} is $\tau$-equivariant for $j=0, \ldots, n$. 
Therefore, the cover $\{Y_j ~:~ 0\leq j\leq n\}$ is a $\tau$-equivariant categorical cover of $M$.
This gives $\ct_{\left <\tau\right>}(M)\leq n+1$.    
\end{proof}

We note that $n+1 \leq \ct_{\left< \tau \right>}(M) \leq \# M^{\Z_2^n}$ for any involution determined by a $\Z_2$-subgroup of $\Z_2^n$.
Consider a small cover $M=M^n(P,\lambda)$. Observe that the column vector of $\lambda$ determines a $\Z_2$-action of $M^n(P,\lambda)$. With this, we have the following result.
\begin{theorem}\label{thm: eq cat}
    Let $M:=M^n(P,\lambda)$ be a small cover. Then 
    \[\ct(M_{F_{i_1}}\cup\dots \cup M_{F_{i_k}}\cup M^{\Z_2^n})\leq \ct_{\Z_2}(M)\leq   \# M^{\Z_2^n},\]
    where $M_{F_{i_j}}$ is the characteristic submanifold corresponding to the facet $F_{i_j}$ for $1\leq j\leq k$ such that $T_{F_{i_j}}=T_{F_{i_l}}$ for $1\leq j\neq l\leq k$.
\end{theorem}
\begin{proof}
  Observe that the fixed point set $M^{\Z_2}=M_{F_{i_1}}\cup\dots \cup M_{F_{i_1}}\cup M^{\Z_2^n}$. Then, the left inequality follows from \cite[Corollary 2.9]{BS}. 
  Let $\mathfrak{q} \colon M \to P$ be an orbit map. Then, one can observe that 
  $\{q^{-1}(U_v) \mid v\in V(P)\}$ forms a $\Z_2$-categorical cover of $M$, where 
  \[\mathfrak{q}^{-1}(U_v)=\frac{\Z_2^n\times U_v}{\sim}\simeq \frac{\Z_2^n\times \{v\}}{(g,v)\sim (h,v)}.\]
  This gives the right inequality.
  \end{proof}

\begin{example}\label{exm: eqcatM}
\begin{enumerate}
    \item  Let $P=I^2$ and consider that facets are labelled by $\{F_1,F_2,F_3,F_4\}$. Define the characteristic function as follows $\lambda(F_1)=(1,0)=\lambda(F_3)$, $\lambda(F_2)=(0,1)$ and $\lambda(F_4)=(1,1)$. 
    Note that $M^n(P,\lambda)=M$ is the Klein bottle.
    Let $\Z_2$-action on $M$ be given by the subgroup $\Z_2\times \{1\}$. Then one can observe that $M^{\Z_2}=M_{F_1}\sqcup M_{F_3}\cong S^1\sqcup S^1$.
    Therefore, from \Cref{thm: eq cat} we get $\ct_{\Z_2}(M)=4$.
    \item One can generalize the above example. Let $P=P_{2m}$ be the $2m$-gon. We label facets of $P_{2m}$ by $\{F_1,F_2,\dots F_{2m-1},F_{2m}\}$. Then note that we can define a characteristic function $\lambda$ such that  $\lambda(F_{2i-1})=(1,0)$ for $1\leq i\leq m$. 
    Then the $\Z_2$-action on $M=M^n(P,\lambda)$ determined by the $\Z_2$-subgroup $\Z_2\times \{1\}$ has fixed point
    $M^{\Z_2}=\cup_{i=1}^mM_{F_{2i-1}}$. Note that $M^{\Z_2}$ contains all vertices of $P_{2m}$ and it is homeomorphic to disjoint union of $m$-many circles as $M_{F_i}\cong S^1$ for each $1\leq i\leq m$. 
    Therefore, \Cref{thm: eq cat} gives us 
    \[2m=\ct(\sqcup_{m}S^1)\leq \ct_{\Z_2}(M)\leq 2m.\]
    Therefore, $\ct_{\Z_2}(M)=2m$.
    
\end{enumerate}
   
\end{example}

The following result is a consequence of \cite[Corollary 5.8]{CG} and \Cref{thm: eqctsmallcover}.
\begin{corollary}
 Let $M$ be an $n$-dimensional small cover such that $M^{\left<\tau \right>}$ is path-connected for some involution $\tau$ on $M$ and the $\left < \tau \right >$-action commutes with the $Z_2^n$-action on $M$. Then, 
$$\TC_{\left< \tau \right>}(M)\leq 2n+1.$$  
\end{corollary}

Next, we discuss the main objectives of this section. The projective product spaces were introduced by Davis in \cite{Davis} as follows:
\[P(p_1, \ldots, p_r)= \frac{S^{p_1}\times \cdots \times S^{p_r}}{(x_1, \dots, x_{r})\sim (-x_1,\dots, -x_{r})}.\] Davis computed the mod-$2$ cohomology algebra of this space.
Recall that an $n$-dimensional real torus manifold admits $\Z_2^n$-action such that the orbit space of this action is a nice manifold with corners. Then one can note that any $\Z_2$-subgroup of $\Z_2^n$ induces an involution on $M$. Now, we consider the following identification spaces. Define 
\[D(M; p_1,\dots,p_r) :=\displaystyle\frac{ M \times S^{p_1}\times\dots \times S^{p_r} }{(y,x_1\dots,x_r)\sim (\tau(y),-x_1,\dots,-x_r)},\] 
where $\tau \colon M\to M$ is an involution.
Then $D(M; p_1,\dots,p_r) $ is a manifold of dimension $n+p_1+\cdots + p_r$. We call this manifold a Dold manifold of real torus type. 
Note that we have a fiber bundle 
\[M \hookrightarrow{}D(M, p_1,\dots,p_r)\stackrel{\mathfrak{p}}{\longrightarrow}P(p_1,\dots,p_r),\] where $P(p_1,\dots,p_r)$ is the projective product space with $p_1 \leq \cdots \leq p_r$.

The following result describes the mod-$2$ cohomology ring of $D(M; p_1, \dots, p_r)$. Let $P$ be a retractable nice manifold with corners in the sense of \cite[Definition 3.2]{Sar}, and $M^n(P, \lambda)$ a real torus manifold. Then $M^n(P, \lambda)$ has a $\Z_2^n$-equivariant cell-structure. Thus, the associated $\Z_2^n$ acts trivially on $H^*(M^n(P, \lambda);\Z_2)$. We note that any simple polytope and the polytope in \Cref{ex: susp pol} are retractable nice manifolds with corners.

\begin{proposition}[{\cite[Proposition 4.6]{SZ22}}]\label{thm CringDMpps}
Let $p_1 \leq \cdots \leq p_r$ and $M$ be a real torus manifold over a retractable $P$. Then the cohomology ring $H^*(D((M,p_1,\dots,p_r);\Z_2)$ is isomorphic as a graded $\Z_2$-algebra to  
 \[H^*(D(M, p_1, \ldots, p_r)); \Z_2) \cong \Z_2[\alpha]/(\alpha^{p_1+1}) \otimes \Lambda[\alpha_2,\dots,\alpha_r] \otimes H^*(M ; \Z_2),\] 
 where $|\alpha| = p_1$, $|\alpha_i| = p_i$ for $i > 1$, and $\Lambda$ denotes the mod-$2$-exterior algebra. 
\end{proposition}

We now compute the LS-category of $D(M; p_1, \dots, p_r)$.

\begin{theorem}\label{thm: ctDmp}
Let $M$  be a real torus manifold over a retractable $P$ and $p_1 \leq \cdots \leq p_r$. Then 
\begin{equation}\label{eq: ctDMrealtmfd}
\ct(D(M; p_1, \dots, p_r)) \geq  \cl_{\Z_2}(M) + p_1+r.  
\end{equation}
Moreover, if $M$ is a small cover of dimension $n$  with $M^{\left< \tau\right >}$ is path connected, then \[\ct(D(M; p_1, \dots, p_r)) =  n + p_1+r.\]
\end{theorem} 
\begin{proof}
  It follows from \cite[Theorem 2.1]{Davis}  that $\cl_{\Z_2}(P(p_1,\dots,p_r))=p_1+r-1$.  
 Therefore, by \Cref{thm CringDMpps}, we get that $\cl_{\Z_2}(D(M; p_1, \dots, p_r))=\cl_{\Z_2}(M)+\cl_{\Z_2}(P(p_1,\dots,p_r))$. Thus, the inequality in \eqref{eq: ctDMrealtmfd} follows from the cup-length lower bound on category.

 Now if $M$ is small cover , then
it follows from \cite[Proposition 2.5]{DSTCgpps} and \cite[Theorem 1.2]{Vandembroucq} 
\[\ct(D(M; p_1, \dots, p_r))\leq q+\ct(P(p_1\dots,p_r))-1\leq q+p_1+r-1,\]
where $q$ is the smallest integer such that $M$ is covered by $q$-many $\tau$-invariant categorical sets.
From \Cref{thm: eqctsmallcover} it follows that $q=n+1$.
This proves the claim.   
\end{proof}

Next, we obtain some bounds for $\TC(D(M; p_1, \dots, p_r))$.
\begin{theorem}\label{thm: tcDmp}
Let  $p_1 \leq \cdots \leq p_r$ and $M$ be a real torus manifold over a retractable $P$ such that $M^{\left<\tau \right>}$ is path-connected for some involution $\tau$ on $M$ and the $\left < \tau \right >$-action commutes with the $Z_2^n$-action on $M$. Then
\begin{equation}\label{eq: tcDMpps}
\zl_{\Z_2}(M)+\zl_{\Z_2}(\R P^{p_1})+r \leq \TC(D(M; p_1, \dots, p_r)) \leq 2(\ct_{\left<\tau \right>}(M) + p_1+r)-1.  
\end{equation}  
Moreover, $$\TC(D(M; p_1, \dots, p_r)) \leq 2(n + p_1+r)+1$$ if $M$ is an $n$-dimensional small cover.
\end{theorem}
\begin{proof}
Note that, using \Cref{thm CringDMpps} we have 
 \[\zl_{\Z_2}((D(M, p_1, \ldots, p_r))=\zl_{\Z_2}(M)+ \zl_{\Z_2}(\R P^{p_1}) +r-1.\]
 Therefore, we get the left inequality of \eqref{eq: tcDMpps}. The right inequality of \eqref{eq: tcDMpps} then follows from \cite[Theorem 2.5]{DSTCgpps}. The remaining follows from \Cref{thm: eqctsmallcover}.
\end{proof}

\begin{corollary}\label{cor: tcdm}
Let $M$ be a real torus manifold satisfying the hypotheses in \Cref{thm: tcmpchi}. Then, 
\begin{equation}
\zl_{\Z_2}(\R P^n)+\zl_{\Z_2}(\R P^{p_1})+1 \leq \TC(D(M; p_1)) \leq 2(n + p_1)+1.  
\end{equation}
 if $n=2^{s}$ and $p_1=2^{t}$, then $$2^{s+1} + 2^{t+1} -1 \leq\TC(D(M;p_1))\leq 2^{s+1}+2^{t+1}+1.$$
\end{corollary}

One can also improve the upper bound for $\TC(D(M; p_1))$ following the proof \Cref{thm_some_odd_fact}.

\begin{proposition}
Let $M$ be a small cover over $\prod_{i=1}^m\Delta^{n_i}$ such that $n_i\ge 1$. If $n_i$'s and $p_1$ are odd then $$\TC(D(M;p_1))\leq 2(n+p_1)-m.$$ 
\end{proposition}

\begin{example}
Consider $S^n$ as a real toric manifold as in Example \ref{ex: susp pol} for $n \geq 2$. Then $\ct_{\left< \tau \right>}(S^n) =2$. Also, $\zl_{\Z_2}(S^{2k+1})=1$ and $\zl_{\Z_2}(S^{2k})=2$ for any $k \in \mathbb{N}$. Thus,  $$\zl_{\Z_2}(S^n) +\zl_{\Z_2}(\R P^{p_1})+1 \leq \TC(D(M; p_1)) \leq 2p_1+5.$$ If $p_1$ is the power of $2$, then  $$2p_1+r+ \zl_{\Z_2}(S^n) \leq \TC(D(M; p_1, \dots, p_r)) \leq 2p_1+2r+3.$$
\end{example}

We now provide some applications of our results to compute the symmetric topological complexity $\TC^S(Y)$ of several real Torus manifolds $Y$.
In \cite{FG07}, Farber and Grant studied the symmetric analog of the motion planning problem and introduced the notion of symmetric topological complexity. Let $N_Y$ be the sub-ring of $H^*(Y) \otimes H^*(Y)$ spanned by the elements of the form $x \otimes y+ y \otimes x$ with $x \neq y$. The following result follows from Corollary $9$, Proposition $10$, and Theorem $17$ in \cite{FG07}. 

\begin{proposition} \label{lower_upper_bound_symm_tc}
    Let $Y$ be a closed smooth manifold. Then 
    $$ \mbox{max} \{ \TC(Y), \cl(N_Y)+ 2 \} \leq \TC^S(Y) \leq 2 \mathrm{dim}(Y)+ 1. $$
\end{proposition}

\begin{remark} \label{cup_length_zcl}
The elements $1 \otimes y_j+ y_j \otimes 1$ and $1 \otimes y_j- y_j \otimes 1$ in $N_{M^n(P, \lambda)}$ are the same, since we are considering the mod-$2$ cohomology ring of real torus manifolds. Consequently, the zero-divisors-cup-length of $M^n(P, \lambda)$ coincides with the cup-length of $N_{M^n(P, \lambda)}$.  
\end{remark}

\begin{theorem}\label{thm: symtcmpchi}
Let $M^n(P, \lambda)$ be a small cover over $P= \prod_{j=1}^m \Delta^{n_j} $. If  $n_j \leq 2^{r_j}-1 < 2n_j$, then $$\TC^S(M^n(P, \lambda)) \geq (2^{r_1}+ \cdots + 2^{r_m})- m+2.$$
In particular, if $n_j=2^{r_j-1}$ for $j=1,2,\dots,m$, then $ \TC^S(M^n(P, \lambda))\geq 2n-m+2.$
\end{theorem}
\begin{proof}
    This follows from \Cref{highertc_small_cover_special}, \Cref{lower_upper_bound_symm_tc}, and Remark \ref{cup_length_zcl}.
\end{proof}

\begin{theorem}\label{thm: sytcexact}
Let $M^n(P,\lambda)$ be a real torus manifold satisfying the hypotheses in \Cref{thm: tcmpchi}. Then, 
\begin{equation*}
  \zl_{\Z_2}(\R P^{n}) +2\leq \TC^S(M^n(P,\lambda))\leq 2n +1.
\end{equation*}
In particular, if $n=2^s$, then $$\TC^S(M^n(P,\lambda))= 2n+1.$$
\end{theorem}
\begin{proof}
This follows from \Cref{thm: tcmpchi}, \Cref{lower_upper_bound_symm_tc} and Remark \ref{cup_length_zcl}.
\end{proof}

\begin{theorem}\label{thm: symtcdmp}
Let $M^n(P,\lambda)$ be a real torus manifold satisfying the hypotheses in \Cref{thm: tcmpchi}. Then, 
\begin{equation*}
\zl_{\Z_2}(\R P^n)+\zl_{\Z_2}(\R P^{p_1})+2 \leq \TC^S(D(M^n(P, \lambda); p_1)) \leq 2(n + p_1)+1.  
\end{equation*}
In particular, if $n=2^{s}$ and $p_1=2^{t}$, then $$2^{s+1} + 2^{t+1} \leq\TC^S(M^n(P,\lambda))\leq 2^{s+1}+2^{t+1}+1.$$
\end{theorem}
\begin{proof}
 This follows from \Cref{thm: tcmpchi},  \Cref{lower_upper_bound_symm_tc} and Remark \ref{cup_length_zcl}.
\end{proof}

\vspace{.5cm}

\noindent{\bf Acknowledgement.} The authors would like to thank Bikramaditya Naskar and Subhankar Sau for the helpful discussion. Koushik Brahma thanks the Indian Statistical Institute Kolkata and Chennai Mathematical Institute for their financial support. Navnath Daundkar thanks NBHM for the support through the grant 0204/10/(16)/2023/R\&D-II/2789. The third author thanks SERB (Now ANRF) India for the CRG Grant (CRG/2023/000239).

\bibliographystyle{abbrv}

\bibliography{References}

\end{document}